\documentclass[a4paper]{amsart}

\usepackage{amssymb,amsmath,amsthm}
\usepackage{mathrsfs}
\usepackage{graphicx}
\usepackage{wrapfig}
\usepackage{caption}
\usepackage{subcaption}
\usepackage{enumerate}
\usepackage[all]{xypic}
\usepackage{multicol}
 \usepackage{url}

\newcommand{\maxx}{\textup{\emph{M}\thinspace}}
\newcommand{\tax}{\textup{\emph{T}\thinspace}}
\newcommand{\cax}{\textup{\emph{C}\thinspace}}
\newcommand{\kax}{\textup{\emph{K}\thinspace}}
\newcommand{\four}{\textup{\emph{4}\thinspace}}
\newcommand{\true}{\textup{\emph{N}\thinspace}}
\newcommand{\dax}{\textup{\emph{D}\thinspace}}

\newcommand{\sfour}{\textup{$\mathbf{S4}$}\thinspace}
\newcommand{\mntf}{\textup{$\mathbf{MNT4}$}\thinspace}

\newcommand{\mtf}{\textup{$\mathbf{MT4}$}\thinspace}
\newcommand{\cpc}{\textup{$\mathbf{CPC}$}\thinspace}

\newcommand{\giez}{\textup{$\mathbf{G}_{0}$}\thinspace}
\newcommand{\giej}{\textup{$\mathbf{G}_{1}$}\thinspace}

\newcommand{\nec}{\textup{\emph{RN}\thinspace}}
\newcommand{\monot}{\textup{\emph{RM}\thinspace}}
\newcommand{\mpon}{\textup{\emph{MP}\thinspace}}
\newcommand{\rext}{\textup{\emph{RE}\thinspace}}

\newcommand{\gt}{\textup{\textbf{GT}\thinspace}}
\newcommand{\gtf}{\textup{\textbf{GT$\mathcal{F}$}\thinspace}}
\newcommand{\gts}{\textup{\emph{GTS}\thinspace}}
\newcommand{\sgts}{\textup{\emph{sGTS}\thinspace}}
\newcommand{\sgt}{\textup{\textbf{sGT}\thinspace}}

\newcommand{\gtn}{\textup{\textbf{GTn}\thinspace}}

\newcommand{\bgtf}{\textup{\textbf{bGT$\mathcal{F}$}\thinspace}}

\newcommand{\gtff}{\textup{\textbf{GT$\mathbf{f}$}\thinspace}}
\newcommand{\gtfi}{\textup{\textbf{GT$\mathbf{f2}$}\thinspace}}

\title[Generalized topological semantics for weak modal logics]{Generalized topological semantics \\ for weak modal logics}
\author{Tomasz Witczak}

\address{Institute of Mathematics\\ University of Silesia\\ Bankowa~14\\ 40-007 Katowice\\ Poland}

\email{tm.witczak@gmail.com}

\date{}

\theoremstyle{Theorem}
\newtheorem{tw}{Theorem}[section]

\theoremstyle{Lemma}
\newtheorem{lem}[tw]{Lemma}

\theoremstyle{Remark}

\theoremstyle{Remark}

\theoremstyle{Definition}
\newtheorem{df}[tw]{Definition}

\theoremstyle{Remark}

\usepackage{graphicx}

\begin{document}

\maketitle

\begin{abstract}
Generalized topological spaces are not necessarily closed under finite intersections. Moreover, the whole universe does not need to be open. We use modified version of this framework to establish certain models for non-normal modal logics. We consider at least two approaches to this topic, wherein one of them is based on the interplay of two operators, both reflecting the idea of necessity. The second one leads to the quite known logics \mtf and \mntf. We obtain some completeness results and we prove correspondence between different classes of frames. Also, three types of bisimulation are investigated.
\end{abstract}

\section{Introduction}

Topological notions are very useful in formal logic. They form bridge between possible-world semantics (which can be considered as somewhat abstract) and well-known mathematical objects (like real line, real plane or Cantor set). Topology allows us to discuss various properties of possible-world frames (depending on axioms of separation or on the notions of density, compactness etc.). Moreover, sometimes these properties can be characterized by means of specific formulas. 

On the other hand, topology is rather strong notion. For example, topological semantics for modal logics leads us to systems not weaker than \sfour. They are equivalent with the so-called \sfour \emph{neighbourhood frames} (see \cite{pacuit}). However, neighbourhoods are most frequently used with non-normal logics, sometimes very weak. The problem is that in topology \emph{neighbourhood} is very rigorous notion, while in possible-worlds semantics it is just an arbitrary (maybe even empty) collection of these worlds which are \emph{assigned} to the given world $w$. 

For this reason, it is difficult to speak about topological semantics for logics weaker than \sfour, not to mention non-normal systems. Hence, we use the concept of \emph{generalized} topological spaces analysed by Cs\'{a}sz\'{a}r in \cite{csaszar}. It is quite natural: the author discarded superset axiom (i.e. the whole universe may not be open) and he assumed that his family is not closed under finite intersections. 

One should be aware that Cs\'{a}sz\'{a}r's spaces are not his own invention. Rather, he rediscovered them, starting their systematical study (then continued by the authors from all over the world). In fact, they have been known earlier as \emph{interior systems} (closely related to \emph{Moore families}). In formal concept analysis and data clustering they are investigated as \emph{extensional abstractions} (see \cite{soldano}) or \emph{knowledge spaces} (see \cite{doignon}). Be as it may, in the past twenty years there have been developed generalized analogues of many topological notions: separation axioms, (nowhere) density, continuity, convergence or topological groups. 

Moreover, there are also further generalizations of the notion of topology: pretopologies, peritopologies, 
(generalized) weak structures or minimal structures. It seems that their logical applications are still on the initial stage, at least in the context of possible worlds semantics. Some interesting results for peritopological spaces (e.g. about modal (in)definability of separation axioms) have been obtained by Ahmet and Terziler in \cite{ahmet}. 

\section{Alphabet and language}

Speaking about logic, we use rather standard language:

\begin{enumerate}
\item $PV$ is a fixed denumerable set of propositional variables $p, q, r, s, ...$
\item Logical connectives and operators are $\land$, $\lor$, $\rightarrow$, $\bot$, $\lnot$ and $\Box$.
\end{enumerate}

Formulas are generated recursively in a standard manner: if $\varphi$, $\psi$ are \emph{wff's} then also $\varphi \lor \psi$, $\varphi \land \psi$, $\varphi \rightarrow \psi$ and $\Box \varphi$. Attention: $\Leftarrow, \Rightarrow$ and $\Leftrightarrow$ are used only on the level of meta-language.

We shall work with the following list of axioms schemes and rules:

{\small
\begin{multicols}{2}
\begin{itemize}

\item \maxx: $\Box(\varphi \land \psi) \rightarrow \Box \varphi \land \Box \psi$
\item \cax: $\Box \varphi \land \Box \psi \rightarrow \Box(\varphi \land \psi) $
\item \tax: $\Box \varphi \rightarrow \varphi$
\item \dax: $\Box \rightarrow \lnot \Box \lnot \varphi$
\item \kax: $\Box (\varphi \rightarrow \psi) \rightarrow (\Box \varphi \rightarrow \Box \psi)$
\item \four: $\Box \varphi \rightarrow \Box \Box \varphi$
\item \true: $\Box \top$ (\emph{truth axiom})
\item \rext: $\varphi \leftrightarrow \psi \vdash \Box \varphi \leftrightarrow \Box \psi$ (rule of extensionality)
\item \nec: $\varphi \vdash \Box \varphi$ (rule of necessity)
\item \monot: $\varphi \rightarrow \psi \vdash \Box \varphi \rightarrow \Box \psi$ (rule of monotonicity)
\item \mpon: $\varphi, \varphi \rightarrow \psi \vdash \psi$ (\emph{modus ponens})

\end{itemize}
\end{multicols}}

Later we shall discuss validity of these formulas in various semantical settings.

\section{Generalized topological spaces}

Here we recall the concept of Cs\'{a}sz\'{a}r but with our own notation which is adapted to our further logical considerations. 

\begin{df}
\label{gentop}

Assume that there is given a non-empty set (universe) $W$. We say that $\mu \subseteq P(W)$ is a generalized topology on $W$ \emph{iff}:

\begin{enumerate}

\item $\emptyset \in \mu$.

\item If $J$ is an arbitrary non-empty set and for each $i \in J$, $X_i \in \mu$, then $\bigcup_{i \in J} X_i \in \mu$.

\end{enumerate}
\end{df}

We say that $\mu$ is \emph{strong} if (and only if) $W \in \mu$. We denote such space by \sgts. Any member of $\mu$ is called $\mu$-\emph{open set}. Later we shall speak just about \emph{open sets}. Any possibility of misunderstanding will be signalized. We define $\mu$-\emph{interior} of $X \subseteq W$ as the greatest open set contained in $X$ (or equivalently as the union of all open sets contained in $X$). We denote it by $Int (X)$. Finally, we say that $(W, \mu)$ is a \emph{generalized topological space} (\gts, g.t.s. or gen.top.).

Below we present several examples of (finite and infinite) \gt-spaces (taken from \cite{sarsak}, \cite{baskar} and \cite{zand}):
{\small
\begin{enumerate}
\item $W = \{a, b, c\}$, $\mu = \{\emptyset, \{a\}, \{b\}, \{a, b\}\}$.

\item $W = \{a, b, c\}$, $\mu = \{\emptyset, \{a\}, \{c\}, \{a, b\}, \{a, c\}, \{b, c\}, W\}$.

\item $W = \mathbb{R}$, $\mu = \{A \subseteq \mathbb{R}; A \subseteq \mathbb{R} \setminus \mathbb{Z}\}$ (so-called $\mathbb{Z}$\emph{-forbidden} \gt).

\item $W$ is arbitrary, $\emptyset \neq X \subseteq W$, $\mu = \{A \subseteq W; A \subseteq W \setminus X\}$. This is generalization of the preceding example.

\item $W = \mathbb{Z}$, $\mu = \{\emptyset, \{1\}, \{1, 3\}, \{1, 3, 5\}, \{1, 3, 5, 7\}, ... \}$.

\item $W = \mathbb{R}$, $\mu = \{\emptyset\} \cup \{A \subseteq \mathbb{R}; A \setminus \{w\} \subseteq A$ for certain $w \in \mathbb{R}\}$.

\item $|W| > \omega_{0}$, $v \notin W$, $W^{*} = W \cup \{v\}$, $\mu = \{\emptyset, \{v\} \cup \{W \setminus A\}; A \subseteq W, |A| \leq \omega_{0}\}$, our space is $\langle W^{*}, \mu \rangle$. 

\item $W = \mathbb{R}$, $\Lambda = \{ [a, b]; a, b \in \mathbb{R}, a < b\}$, $\mu$ is a collection of all unions which belong to $\Lambda$. 
\end{enumerate}}

As we have already said, many topological notions have their generalized counterparts. It seems that in many cases \gt-spaces behave in a manner identical with that of ordinary spaces. However, there are some situations in which the whole case becomes less trivial. For example, a group of Polish authors (\cite{korczak}) pointed out that two typical definitions of \emph{nowhere density}, which are equivalent in a standard framework, are not equivalent in \gt-spaces. This allows us to distinguish between \emph{nowhere dense sets} (for which $Int(Cl(A)) = \emptyset$) and \emph{strongly nowhere dense sets} (for any $G \in \mu$ there is $H \in \mu$ such that $H \subseteq G$ and $A \cap H = \emptyset$). Our authors used this distinction in their research about Baire spaces and Banach games. Of course this is completely beyond the scope of our present paper but we would like just to emphasize the fact that generalizations of topology have their internal value. 

\subsection{Generalized topological models with $\mathcal{F}$}

In this section we show certain modal logic framework which is based on the notion of \gt-space. In fact, we extend the basic notion even further. Let us take a look at the structure of our space (which is not necessarily strong). We see that there is a maximal open set $\bigcup \mu$. Hence, there are two kinds of points: those which are in $\bigcup \mu$ and those which are beyond this set. The latter points are in some sense \emph{orphaned}. Recall that in topological semantics we customarily use open neighbourhoods to speak about forcing of modal formulas. Hence, we propose to associate each orphaned point with certain family of open sets by means of a special function $\mathcal{F}$. Here we introduce formal definition of our new structure:

\begin{df}
\label{gtmod}
We define \gtf-model as a quadruple $M_\mu = \langle W_\mu, \mu, \mathcal{F}, V_\mu \rangle$ where $\mu$ is a generalized topology on $W_\mu$, $V_\mu$ is a function from $PV$ into $P(W)$ and $\mathcal{F}$ is a function from $W_\mu$ into $P(P((\bigcup \mu))$ such that:

\begin{itemize}
\item If $w \in \bigcup \mu$, then $[X \in \mathcal{F}_{w} \Leftrightarrow X \in \mu \text{ and } w \in X]$ {\normalfont [$\mathcal{F}_{w}$ is a shortcut for $\mathcal{F}(w)$]}.
\item If $w \in W_\mu \setminus \bigcup \mu$, then $[X \in \mathcal{F}_{w} \Rightarrow X \in \mu]$.
\end{itemize}
\end{df}

$\mathcal{F}$ can be considered as an arbitrary extension (on the whole $W$) of the map that associates with any $w \in \bigcup \mu$ the principal filter generated by this $w$. 

If there is no valuation established, then we say that $\langle W_\mu, \mu, \mathcal{F} \rangle$ is a \gtf-frame. We use the following symbol: $A^{-1} = \{z \in W; A \in \mathcal{F}_{z}\}$ for any $A \in \mu$. 

\begin{df}
If $M = \langle W_{\mu}, \mu, V_\mu \rangle$ is an \gtf-model, then we define relation $\Vdash_\mu$ between worlds and formulas in the following way:

\begin{enumerate}
\item $w \Vdash_\mu q \Leftrightarrow w \in V_\mu(q)$ for any $q \in PV$.
\item $w \Vdash_\mu \varphi \land \psi$ (resp. $\varphi \lor \psi$) $\Leftrightarrow w \Vdash_\mu \varphi$ and (resp. or) $w \Vdash_\mu \psi$.
\item $w \Vdash_\mu \varphi \rightarrow \psi \Leftrightarrow w \nVdash_\mu \varphi$ or $w \Vdash_\mu \psi$.
\item $w \Vdash_\mu \lnot \varphi \Leftrightarrow w \nVdash_\mu \varphi$.
\item $w \Vdash_\mu \Box \varphi \Leftrightarrow \text{ there is } \mathcal{O}_{w} \in \mathcal{F}_{w} \text{ such that for each } v \in \mathcal{O}_{w}, v \Vdash \varphi$.
\end{enumerate}
\end{df} 

Now let us define generalized topological neighbourhoods in this framework. 

\begin{df}
\label{gennei}
If $\langle W_\mu, \mu, \mathcal{F} \rangle$ is a \gtf-frame, then for each $w \in W_\mu$ we define its family of generalized topological neighbourhoods as:

$\mathcal{N}^{\mu}_{w} = \{X \subseteq \bigcup \mathcal{\mu} \text{ such that  there is } O_w \in \mathcal{F}_{w} \text{ that } \mathcal{O}_{w} \subseteq X\}$.
\end{df}


In this paper we use \gtf-structures merely as logical models. However, we have developed a small theory of these structures, which was presented in \cite{asso}. In particular, we have introduced and studied several topological notions like $\mathcal{F}$-interior (-closure), $\mathcal{F}$-convergence (of generalized nets) and some separation axioms. 

\subsection{Suitable neighbourhood frames}

Now let us think about corresponding neighbourhood frames. Our goal is to obtain pointwise equivalence between them and \gtf-structures.

\begin{df}
\label{topmod}
We define \gtn-model as a triple $M = \langle W, \mathcal{N}, V \rangle$ where $V$ is a function from $PV$ into $P(W)$, $\mathcal{N}$ is a function from $W$ into $P(P(W))$ and:

\begin{enumerate}

\item $\bigcup \mathcal{N}$ is a union of all sets $X$ for which there is $w \in W$ such that $X \in \mathcal{N}_{w}$. {\normalfont [Namely, $\bigcup \mathcal{N}$ is a union of \emph{all} neighbourhoods]}.

\item $W = W_{1} \cup W_{2}$, where $W_{1} = \{z \in W; z \in \bigcap \mathcal{N}_{z}\}$ and $W_{2} = \{z \in W; z \notin \bigcup \mathcal{N}\}$.

\item $X \in \mathcal{N}_{w}$ and $X \subseteq Y \subseteq \bigcup \mathcal{N} \Rightarrow Y \in \mathcal{N}_{w}$.

\item $X \in \mathcal{N}_{w} \Rightarrow \{z \in W_1; X \in \mathcal{N}_{z}\} \in \mathcal{N}_{w}$.

\end{enumerate}
\end{df}

Note that we have two kinds of worlds. Those which are in \emph{certain} neighbourhood, are also in each of \emph{their own} neighbourhoods. As for the worlds from $W_{2}$, they have their neighbourhoods but they do not belong to \emph{any} neighbourhood.

As for the modal forcing, we define it in the following manner:

$w \Vdash \Box \varphi \Leftrightarrow \text{ there is } X \in \mathcal{N}_{w} \text{ such that } X \subseteq V(\varphi)$.


\subsection{From neighbourhoods to general topologies}

Let us take our new structure and introduce topology into it. We propose the following definition:

\begin{lem}
\label{mu}
Assume that $M = \langle W, \mathcal{N}, V \rangle$ is a \gtn-model. Then the set $\mu = \{X \subseteq \bigcup \mathcal{N}; w \in X \Rightarrow X \in \mathcal{N}_{w}\}$ forms a generalized topology with $W$ as its universe. Equivalently, $\mu$ forms strong generalized topology with $\bigcup \mathcal{N}$ as its universe.
\end{lem}

\begin{proof}
The proof is simple. Note that we assume that open sets are contained in $\bigcup \mathcal{N}$.
\end{proof}

The next lemma is simple but in some sense crucial:

\begin{lem}
Assume that $M = \langle W, \mathcal{N}, V \rangle$ is a \gtn-model, $\mu = \{X \subseteq \bigcup \mathcal{N}; w \in X \Rightarrow X \in \mathcal{N}_{w}\}$ is a generalized topology and $\bigcup \mu$ is the union of all open sets. Then $\bigcup \mu = \bigcup \mathcal{N}$. 
\end{lem}

\begin{proof} 
$(\subseteq)$ If $w \in \bigcup \mu$, then there is $Y \in \mu$ such that $w \in Y$. But by the very definition of $\mu$, $Y \subseteq \bigcup \mathcal{N}$. 

$(\supseteq)$ Suppose that $w \in \bigcup \mathcal{N}$. Then $w \in \bigcap \mathcal{N}_{w}$. Let us take any $X \in \mathcal{N}_{w}$. Then $X \subseteq \bigcup \mathcal{N}$ and $w \in X$, so $X$ is $\mu$-open. Thus $w \in \bigcup \mu$.
\end{proof}

Now we can go to our transformation:

\begin{tw}
Let $M = \langle W, \mathcal{N}, V \rangle$ be a \gtn-model. Then there exists pointwise equivalent \gtf-model $M_\mu = \langle W_\mu, \mu, \mathcal{F}, V_\mu \rangle$. 

\end{tw}

\begin{proof}
Let $W_\mu = W, V_\mu = V$ and for every $w \in W$, $W \supseteq X \in \mathcal{F}_{w} \Leftrightarrow X \in \mathcal{N}_{w}$. Define $\mu$ as in Lemma \ref{mu}. Then $\bigcup \mu = \bigcup \mathcal{N}_{w}$. Hence, $\mathcal{F}$ is indeed function from $W_\mu$ into $P(P(\bigcup \mu))$. 

Now suppose that $w \in \bigcup \mu$ and $X \in \mathcal{F}_{w}$. Hence, $X \in \mathcal{N}_{w}$, $w \in X$ and thus $X \in \mu$. On the other side, if $w \in X$ and $X \in \mu$, then $X \in \mathcal{N}_{w} = \mathcal{F}_{w}$. Now let us take $w \in W \setminus \bigcup \mu = W \setminus \bigcup \mathcal{N}$ and suppose that $X \in \mathcal{F}_{w} = \mathcal{N}_{w}$. Then clearly $X \subseteq \bigcup \mu$ and meta-implication "$w \in X \Rightarrow X \in \mathcal{N}_{w}$" becomes true because it is trivial. Thus we checked all expected properties of $\mathcal{F}$.

Let us define $\mathcal{N}^{\mu}_{w}$ for each $w \in W$ like in Def. \ref{gennei}. Then $\mathcal{N}^{\mu}_{w} = \mathcal{N}_{w}$.

$(\subseteq)$ Assume that $X \in \mathcal{N}^{\mu}_{w}$. Then $X \subseteq \bigcup \mu$. Hence, $X \subseteq \bigcup \mathcal{N}$. Moreover, there is $\mathcal{O}_{w} \in \mathcal{F}_{w}$ such that $\mathcal{O}_{w} \subseteq X$. Hence, $\mathcal{O}_{w} \in \mathcal{N}_{w}$. But then, by means of restriction 3 from Def. \ref{topmod}, $X \in \mathcal{N}_{w}$. 

$(\supseteq)$ Assume that $X \in \mathcal{N}_{w}$. Then $X \subseteq \bigcup \mathcal{N} = \bigcup \mu$. Let us consider $\mathcal{O}_{w} = \{v \in \bigcup \mathcal{N}; X \in \mathcal{N}_{v}\}$. Because of restriction 4 from Def. \ref{topmod}, $\mathcal{O}_{w} \in \mathcal{N}_{w}$, hence $\mathcal{O}_{w} \in \mathcal{F}_{w}$. If $z \in \mathcal{O}_{w}$ (which implies in particular that $z \in \bigcup \mathcal{N}$), then $z \in \bigcap \mathcal{N}_{z} \subseteq X$. Thus $\mathcal{O}_{w} \subseteq X$. 

Now we can think about pointwise equivalency. We are concentrated only on the modal case. Assume that $w \Vdash \Box \varphi$. Hence, there is $X \in \mathcal{N}_{w}$ such that for each $v \in X, v \Vdash \varphi$. From earlier considerations we have that $X \in \mathcal{N}^{\mu}_{w}$, so there is (open) $\mathcal{O}_{w} \in \mathcal{F}_{w}$ such that $\mathcal{O}_{w} \subseteq X \subseteq V(\varphi) = V_\mu(\varphi)$. Thus, $w \Vdash_\mu \Box \varphi$. The other direction, starting from the assumption that $w \Vdash_\mu \Box \varphi$, is similar.

\end{proof}

The next theorem states that reverse transformation is also possible.

\begin{tw}
Let $M_\mu = \langle W_\mu, \mu, \mathcal{F}, V_\mu \rangle$ be a \gtf-model. Then there exists pointwise equivalent \gtn-model $M = \langle W, \mathcal{N}^{\mu}, V \rangle$. 

\end{tw}

\begin{proof}

As earlier, assume that $W = W_\mu$ and $V_\mu = V$. We must establish neighbourhoods in our topological setting. But we can do it exactly in the same way as in the preceding theorems and definitions (recall Def. \ref{gennei}). Now me must check that $\mathcal{N}^{\mu}$ satisfies all the properties of neighbourhoods in the sense of \gtn-frames. 

First, let us show that $\bigcup \mu = \bigcup \mathcal{N}^{\mu}$. 

($\subseteq$): If $w \in \bigcup \mu$, then there is $\mathcal{O}_{w} \in \mu$ such that $w \in \mathcal{O}_{w}$. In particular, $\mathcal{O}_{w} \in \mathcal{F}_{w}$. Hence, by means of Def. \ref{gennei}, $\mathcal{O}_{w} \in \mathcal{N}^{\mu}_{w}$. 

($\supseteq$): Assume that $w \in \bigcup \mathcal{N}^{\mu}$. Hence, there are $v \in W, X \in \mathcal \mathcal{N}^{\mu}_{v}$ such that $w \in X$. But by the definition, $X \subseteq \bigcup \mu$. 

Now suppose that $w \in W, X \in \mathcal{N}^{\mu}_{w}$ and $X \subseteq Y \subseteq \bigcup \mathcal{N}^{\mu}$. Hence, there is $\mathcal{O}_{w} \in \mathcal{F}_{w}$ such that $\mathcal{O}_{w} \subseteq X \subseteq \bigcup \mu$.  But also $Y \subseteq \bigcup \mu$ and of course $\mathcal{O}_{w} \subseteq Y$. Hence, $Y \in \mathcal{N}^{\mu}_{w}$. 

Assume that $X \in \mathcal{N}^{\mu}_{w}$ and consider $S = \{z \in W_1; X \in \mathcal{N}^{\mu}_{z}\}$. There is $\mathcal{O}_{w} \in \mathcal{F}_{w}$ such that $\mathcal{O}_{w} \subseteq X \subseteq \bigcup \mu$. However, $S \subseteq X$: take $v \in S$ and recall the fact that $v \in \bigcap \mathcal{N}^{\mu}_{v}$, hence $v \in X$. Moreover, $S$ contains only (some) points from $W_1$, so $S \subseteq \bigcup \mu$. 

Now we may go to the pointwise equivalency. Suppose that $w \Vdash_\mu \Box \varphi$. Hence, there is $\mathcal{O}_{w} \in \mathcal{F}_{w}$ such that $\mathcal{O}_{w} \subseteq V_\mu(\varphi) = V(\varphi)$ (in the last equation we used induction hypothesis). Then $\mathcal{O}_{w} \in \mathcal{N}^{\mu}_{w}$. Hence, $w \Vdash \Box \varphi$. The other direction, starting from the assumption that $w \Vdash \Box \varphi$, is similar.

\end{proof}

\subsection{Axioms and rules}

Unfortunately, we do not have complete axiomatization of the logic induced by our semantics. Certainly, axioms \maxx and \four are true (i.e. they are satisfied in each world independently of valuation). On the other hand, axiom \tax does not hold: it is easy to build counter-model with at least one world $w \in W \setminus \bigcup \mu$ which satisfies $\Box \varphi$ but does not accept $\varphi$. 

We can list several regularities:

\begin{lem}In each \gt-model $M_\mu = \langle W_\mu, \mu, \mathcal{F}, V_\mu \rangle$ the following holds:

\begin{enumerate}
\item If $w \Vdash_\mu \varphi$ for each $w \in \bigcup \mu$, $v \in W \setminus \bigcup \mu$ and $\emptyset \notin \mathcal{F}_{v}$, then $v \Vdash_\mu \Diamond \varphi$, where $\Diamond \varphi$ is a shortcut for $\lnot \Box \lnot \varphi$. 

\item If $w \Vdash_\mu \varphi$ for each $w \in \bigcup \mu$, $v \in W \setminus \bigcup \mu$ and $\mathcal{F}_{v} \neq \emptyset$, then $v \Vdash_\mu \Box \varphi$.

\item If for each $w \in W$, $\mathcal{F}_{w} \neq \emptyset$ and if $\varphi$ holds in each world from $\bigcup\mu$, then $\Box \varphi$ holds in each world from $W$. Consequently, standard \nec also becomes true with this assumptions. 

\item If $v \in W \setminus \bigcup \mu, w \in \bigcup \mu$ and $\mathcal{F}_{v} = \mathcal{F}_{w}$, then $v \Vdash_\mu \varphi \Leftrightarrow w \Vdash \varphi$ for any $\varphi$ from the set $MOD$ which is built in the following way: \textbf{i)} for each formula $\gamma$, $\Box \gamma \in MOD$, \textbf{ii)} $\alpha \lor \beta, \alpha \land \beta, \alpha \rightarrow \beta$ are all in $MOD$, where $\alpha, \beta$ are already in $MOD$. 

\end{enumerate}
\end{lem}

Of course there are two specific cases in which it is possible to obtain completeness quite easily:

\textbf{1}. The case of strong frames, i.e. those in which $W = \bigcup \mu$. In this situation our \gt-model corresponds to the strong \gtn-model with $W_2 = \emptyset$. As we can conclude from \cite{indrze}, this class coincides with logic \mntf (i.e. there is a completeness result). 

\textbf{2}. The case of frames in which $\mathcal{F}_{w} = \emptyset$ for each $w \in W \setminus \bigcup \mu$. Here we have (see \cite{jarvi}) completeness with respect to the logic \mtf. In this situation we can consider elements of $W \setminus \bigcup \mu$ as so-called \emph{impossible worlds}: worlds in which "everyting is possible and nothing is necessary" (that is, $\Diamond \varphi$ holds for any $\varphi$ and $\Box \varphi$ fails for any $\varphi$). It is noteworthy that in \cite{jarvi} the authors have shown that \mtf can be embedded in their ”Information Logic of Galois Connections” (\textbf{ILGC}). This is a connection between non-normal modal logics and the theory of rough sets or approximate reasoning. 

\section{Generalized topo-bisimulations}

In this section we introduce three notions of generalized topo-bisimulation between two \gtf-models. They are based on the topo-bisimulation presented e.g. by Aiello et al. in \cite{aiello}. However, this standard definition was adapted to our needs.

\begin{df}
Assume that $M_1 = \langle W_\mu, \mu, \mathcal{F}^{\mu}, V_\mu \rangle$ and $M_2 = \langle W_\tau, \tau, \mathcal{F}^{\tau}, V_\tau \rangle$ are two \gtf-models. We define generalized $0$-topo-bisimulation as a non-empty relation $T \subseteq W \times W'$ such that if $w T w'$ (where $w \in W, w' \in W'$), then:

\begin{enumerate}
\item $w \Vdash_\mu q \Leftrightarrow w' \Vdash_\tau q$ for any $q \in PV$.
\item If $w \in \mathcal{O}_{w} \in \mu$, then there is $\mathcal{O}_{w'} \in \tau$ such that for each $v' \in \mathcal{O}_{w'}$ there exists $v \in \mathcal{O}_{w}$ such that $v T v'$.
\item If $w' \in \mathcal{O}_{w'} \in \tau$, then there is $\mathcal{O}_{w} \in \mu$ such that for each $v \in \mathcal{O}_{w}$ there exists $v' \in \mathcal{O}_{w'}$ such that $v T v'$.
\end{enumerate}

\end{df}

Such function is useful mostly for these worlds which are somewhere in $\bigcup \mu$. The next notion is more general:

\begin{df}
Assume that $M_1 = \langle W_\mu, \mu, \mathcal{F}^{\mu}, V_\mu \rangle$ and $M_2 = \langle W_\tau, \tau, \mathcal{F}^{\mu}, V_\tau \rangle$ are two \gtf-models. We define generalized $1$-topo-bisimulation as a non-empty relation $T \subseteq W \times W'$ such that if $w T w'$ (where $w \in W, w' \in W'$), then:

\begin{enumerate}
\item $w \Vdash_\mu q \Leftrightarrow w' \Vdash_\tau q$ for any $q \in PV$.
\item If $w \in \mathcal{O}^{-1}_{w}$, where $\emptyset \neq \mathcal{O}_{w} \in \mu$, then there is $\mathcal{O}_{w'} \in \mathcal{F}^{\tau}_{w'}$ such that for each $v' \in \mathcal{O}_{w'}$ there exists $v \in \mathcal{O}_{w}$ such that $v T v'$.
\item If $w' \in \mathcal{O}^{-1}_{w'}$, where $\emptyset \neq \mathcal{O}_{w'} \in \tau$, then there is $\mathcal{O}_{w} \in \mathcal{F}^{\mu}_{w}$ such that for each $v \in \mathcal{O}_{w}$ there exists $v' \in \mathcal{O}_{w'}$ such that $v T v'$.
\end{enumerate}
\end{df}

The third notion seems to be the most vague:

\begin{df}
Assume that $M_1 = \langle W_\mu, \mu, \mathcal{F}^{\mu}, V_\mu \rangle$ and $M_2 = \langle W_\tau, \tau, \mathcal{F}^{\mu}, V_\tau \rangle$  are two \gtf-models. We define generalized $2$-topo-bisimulation as a non-empty relation $T \subseteq W \times W'$ such that if $w T w'$ (where $w \in W, w' \in W'$), then:

\begin{enumerate}
\item $w \Vdash_\mu q \Leftrightarrow w' \Vdash_\tau q$ for any $q \in PV$
\item If $w \in \mathcal{O}^{-1}_{w}$, where $\emptyset \neq \mathcal{O}_{w} \in \mu$, then there is $\mathcal{O}_{w'} \in \mathcal{F}^{\tau}_{w'}$ such that for each $v' \in \mathcal{O}^{-1}_{w'}$ there exists $v \in \mathcal{O}^{-1}_{w}$ such that $v T v'$.
\item If $w' \in \mathcal{O}^{-1}_{w'}$, where $\emptyset \neq \mathcal{O}_{w'} \in \tau$, then there is $\mathcal{O}_{w} \in \mathcal{F}^{\tau}_{w}$ such that for each $v \in \mathcal{O}^{-1}_{w}$ there exists $v' \in \mathcal{O}^{-1}_{w'}$ such that $v T v'$.
\end{enumerate}
\end{df}

Let us introduce some basic definitions concerning functions. The notions of (ordinary) continuity and openess are taken from \cite{csaszar}:

\begin{df}
Assume that $F_1 = \langle W_\mu, \mu, \mathcal{F}^{\mu} \rangle$ and $F_2 = \langle W_\tau, \tau, \mathcal{F}^{\tau} \rangle$ are two generalized topological frames and $f$ is a function from $W_\mu$ into $W_\tau$. We say that $f$ is:

\begin{itemize}
\item continuous $\Leftrightarrow$ $\left[G' \in \tau \Rightarrow f^{-1}(G') \in \mu \right]$

\item open $\Leftrightarrow$ $f(G) \in \tau \text{ for each } G \in \mu$.

\item $\mathcal{F}$-continuous $\Leftrightarrow$ $\left[G' \in \mathcal{F}^{\tau}_{w'} \Rightarrow f^{-1}(G') \in \mathcal{F}^{\mu}_{w} \right]$ for any $w \in W_\mu, w' \in W_\tau$ such that $f(w) = w'$. 

\item $\mathcal{F}$-open $\Leftrightarrow$ $f(G) \in \mathcal{F}^{\tau}_{w'} \text{ for each } G \in \mathcal{F}^{\mu}_{w}$ for any $w \in W_\mu, w' \in W_\tau$ such that $f(w) = w'$.

\end{itemize}

\end{df}

The following two theorems give us our expected relationship beetween bisimulations and functions introduced above.

\begin{tw}
\label{zerob}
Assume that $F_1 = \langle W_\mu, \mu, \mathcal{F}^{\mu} \rangle$ is a \gtf-frame and $M_2 = \langle W_\tau, \tau, \mathcal{F}^{\tau}, V_\tau \rangle$ is a \gtf-model. Suppose that $f$ is a continuous and open map between $W_\mu$ and $W_\tau$; and for any $q \in PV$ we set $V_\mu(q) = f^{-1}(V_\tau(q))$. Then $f$ is a $0$-topo-bisimulation between $M_1 = \langle W_\mu, \mu, \mathcal{F}^{\mu}, V_\mu \rangle $ and $M_2$. 
\end{tw}

\begin{proof}
We shall present only the first part of the proof. Two other parts will are actually contained in the proof of the next theorem.

Let $f(w) = w'$, where $w \in W_\mu, w' \in W_\tau$.  First, assume that $w \Vdash_\mu q$, so $w \in V_\mu(q) = f^{-1}(V_\tau(q))$. We know that $f(w) = w' \in f(f^{-1}(V_\tau(q))) = V_\tau(q)$. Hence $w' \Vdash_\tau q$. 

Now suppose that $w' \Vdash_\tau q$. Hence, $w' \in V_\tau(q) = f(f^{-1}(V_\tau(q)) = f(V_\mu(q))$. Moreover, $w \in f^{-1}(\{w'\})$. From set theory we know that if $\{w'\} \subseteq f(V_\mu(q))$ (which is true), then $f^{-1}(\{w'\}) \subseteq f^{-1}(f(V_\mu(q)) = V_\mu(q)$. Hence, $w \Vdash_\mu q$. 


\end{proof}

\begin{tw}
Assume that $F_1 = \langle W_\mu, \mu, \mathcal{F}^{\mu} \rangle$ is a \gtf-frame and $M_2 = \langle W_\tau, \tau, \mathcal{F}^{\tau}, V_\tau \rangle$ is a \gtf-model. Suppose that $f$ is a $\mathcal{F}$-continuous and $\mathcal{F}$-open map between $W_\mu$ and $W_\tau$; and for any $q \in PV$ we set $V_\mu(q) = f^{-1}(V_\tau(q))$.  Then $f$ is a $1$-topo-bisimulation between $M_1 = \langle W_\mu, \mu, \mathcal{F}^{\mu}, V_\mu \rangle$ and $M_2$.
\end{tw}

\begin{proof}
Let $f(w) = w'$, where $w \in W_\mu, w' \in W_\tau$. The first part of the proof is exactly the same as in Theorem \ref{zerob}. 

As for the second one, suppose that $w \in \mathcal{O}^{-1}_{w}$, i.e. $\mathcal{O}_{w} \in \mathcal{F}_{w}$. We can say that $f(\mathcal{O}_{w}) \in \mathcal{F}_{w'}$ (because $f$ is $\mathcal{F}$-open). Assume that $v' \in f(\mathcal{O}_{w})$. Of course there must be certain $v \in \mathcal{O}_{w}$ such that $f(v) = v'$. 

Third part is similar. Suppose that $w' \in \mathcal{O}^{-1}_{w'}$, i.e. $\mathcal{O}_{w'} \in \mathcal{F}_{w'}$. We can say that $f^{-1}(\mathcal{O}_{w'}) \in \mathcal{F}_{w}$ (because $f$ is $\mathcal{F}$-continuous). Assume that $v \in f^{-1}(\mathcal{O}_{w'})$. Then $f(v) \in f(f^{-1}(\mathcal{O}_{w'})) = \mathcal{O}_{w'}$. 
\end{proof}

The main benefit of bisimulation is certain kind of "logical similarity" between two models. It has been presented below in two theorems:

\begin{tw}Assume that $M_1 = \langle W_\mu, \mu, \mathcal{F}^{\mu}, V_\mu \rangle$ and $M_2 = \langle W_\tau, \tau, \mathcal{F}^{\mu}, V_\tau \rangle$ are two \gtf-models, $T$ is a $0$-bisimulation between them and there are $w \in \bigcup \mu, w' \in \bigcup \tau$ such that $wTw'$. Then $w$ and $w'$ satisfy the same formulas.
\end{tw}

\begin{proof}
The proof is similar to the one for the next theorem (which is in some sense more general).


\end{proof}

In the next theorem we speak about \emph{consistent} models (frames). It means that each of their worlds is consistent: $\emptyset \notin \mathcal{F}_{w}$.

\begin{tw}Assume that $M_1 = \langle W_\mu, \mu, \mathcal{F}^{\mu}, V_\mu \rangle$ and $M_2 = \langle W_\tau, \tau, \mathcal{F}^{\tau}, V_\tau \rangle$ are two consistent \gtf-models, $T$ is a $1$-bisimulation between them and there are $w \in W_\mu, w' \in W_\tau$ such that $wTw'$. Then $w$ and $w'$ satisfy the same formulas.
\end{tw}

\begin{proof}
The proof goes by induction on the complexity of formulas. If $\varphi := q \in PV$, then our thesis is clear. Boolean cases are simple. 

Assume that $\varphi := \Box \gamma$ and $w \Vdash_\mu \Box \gamma$. There is certain $\mathcal{O}_{w} \in \mu$ such that $\mathcal{O}_{w} \in \mathcal{F}^{\mu}_{w}$ and $\mathcal{O}_{w} \subseteq V_\mu(\gamma)$. Thus, by means of $1$-bisimulation, there is $\mathcal{O}_{w'} \in \mathcal{F}^{\tau}_{w'}$ with expected properties. Let us take $v' \in \mathcal{O}_{w'}$ (we can do it because of consistency of the model). Now we can find $v \in \mathcal{O}_{w}$ such that $vTv'$. By means of induction hypothesis, $v' \Vdash_\tau \gamma$. Then $w' \Vdash_\tau \Box \gamma$.

Now suppose that $w' \Vdash_\tau \Box \gamma$. There is $\mathcal{O}_{w'} \in \tau$ such that $\mathcal{O}_{w'} \in \mathcal{F}^{\tau}_{w'}$ and $\mathcal{O}_{w'} \subseteq V_\tau(\gamma)$. Thus, by means of $1$-bisimulation, there is $\mathcal{O}_{w} \in \mathcal{F}^{\mu}_{w}$ with expected properties. Let us consider $v \in \mathcal{O}_{w}$. Now we can find $v' \in \mathcal{O}_{w'}$ such that $vTv'$. Induction hypothesis allows us to conclude that $v \Vdash_\mu \gamma$. Hence, $w' \Vdash_\mu \Box \gamma$. 

\end{proof}

As for the $2$-bisimulation, we shall not obtain analogous result for $\Box$. However, it is possible to obtain it for new modality, which is in some sense more vague: 

$w \Vdash \bullet \varphi \Leftrightarrow$ there is $\mathcal{O}_{w} \in \mathcal{F}_{w}$ such that for any $v \in \mathcal{O}^{-1}_{w}$, $v \Vdash \varphi$.

Note that if we use this modality, then the axiom scheme $\bullet \varphi \rightarrow \varphi$ becomes true in each \gtf-model. Assume now that we replace $\Box$ by $\bullet$ and let us call corresponding models \bgtf-models. Using this new language, we can prove the theorem below:

\begin{tw}
Assume that $M_1 = \langle W_\mu, \mu, \mathcal{F}^{\mu}, V_\mu \rangle$ and $M_2 = \langle W_\tau, \tau, \mathcal{F}^{\tau}, V_\tau \rangle$ are two consistent \bgtf-models, $T$ is a $2$-bisimulation between them and there are $w \in W_\mu, w' \in W_\tau$ such that $wTw'$. Then $w$ and $w'$ satisfy the same formulas.
\end{tw}

\begin{proof}
The proof goes by induction. We present only the modal case. Assume that $\varphi := \bullet \gamma$ and $w \Vdash_\mu \bullet \gamma$. Hence, there is $\mathcal{O}_{w} \in \mathcal{F}^{\mu}_{w}$ such that $\mathcal{O}_{w}^{-1} \subseteq V_{\mu}(\gamma)$. Using $2$-bisimulation, we find $\mathcal{O}_{w'} \in \mathcal{F}^{\tau}_{w'}$ with expected properties. Now we take $v' \in \mathcal{O}_{w'}^{-1}$ and we find suitable $v \in \mathcal{O}_{w}^{-1}$. Hence, $vTv'$ and for this reason they support the same formulas. In particular, $v' \Vdash_\tau \varphi$. Thus, we can say that $w' \Vdash_\tau \bullet \varphi$. Similar reasoning can be provided if we start from the assumption that $w' \Vdash_\tau \bullet \varphi$.
\end{proof}

At first glance, this new operator seems to be somewhat artificial. However, we can show that if we replace $\Box$ by $\bullet$, then certain subclass of \bgtf-models can be treated as strong \gt-model.

\begin{df}
\label{subclass}
Assume that $F = \langle W, \mu, \mathcal{F} \rangle$ is \bgtf-frame. We say that $M = \langle W, \mu, \mathcal{F}, V_\mu \rangle$ is in-fact-strong (i.f.s.) \bgtf-model iff modality in our language is interpreted as $\bullet$ and the following additional conditions hold for any $w \in W \setminus \bigcup \mu$:

\begin{enumerate}
\item if $X \in \mathcal{F}_{w}$ and $X \subseteq Y \in \mu$, then $Y \in \mathcal{F}_{w}$ (superset condition).
\item if $\bigcup_{i \in J} X_i \in \mathcal{F}_{w}, \text{ where } X_i \in \mu \text{ for any }i \in J$, then there is $k \in J$ such that $w \in X^{-1}_k$ (union partition condition).
\item $\mathcal{F}_{w} \neq \emptyset$.
\end{enumerate}

\end{df}

Now we can formulate the following theorem:

\begin{tw}
For any i.f.s. \bgtf-model there exists pointwise equivalent \sgt-model (with modality interpreted as $\Box$). 
\end{tw}

\begin{proof}
Having generalized topology $\mu$, we must establish new generalized topology $\tau$ in such a way that the whole space becomes $\tau$-open. Hence, we are looking for $M' = \langle W', \tau, V_\tau \rangle$, where we can presuppose that $W' = W$ and $V_\tau = V_\mu$. 

Let us assume that $Y \in \tau \Leftrightarrow Y = \emptyset$ or $Y = X^{-1}$ for certain $X \in \mu$. We must check properties of gen. top. First, $\emptyset \in \tau$ (by the very definition). What about arbitrary unions? Let $J \neq \emptyset$ and $Y_i$ be $\tau$-open for any $i \in J$. We may assume that there is at least one $k \in J$ such that $Y_k \neq \emptyset$. Consider $\bigcup_{i \in J} Y_i = \bigcup_{i \in J} X^{-1}_i$  (where $X^{-1}_{i} = Y_{i}$ and $X_i \in \mu$ for each $i \in J$). We shall show that this set is equal with $(\bigcup_{i \in J} X_i)^{-1}$:

($\subseteq$) Let $w \in \bigcup_{i \in J}Y_i$. Hence, there are $k \in J$ and $Y_k$ such that $w \in Y_k$. Therefore there is $X_k \in \mu$ such that $Y_k = X_k^{-1}$. Thus $X_k \in \mathcal{F}_{w}$. Moreover, $X_k \subseteq \bigcup_{i \in J}X_i$. But from the basic properties of gen. top. $\bigcup_{i \in J} X_{i} \in \mu$. By the superset condition we have that $\bigcup_{i \in J} X_i \in \mathcal{F}_{w}$. Thus $w \in (\bigcup_{i \in J} X_i)^{-1}$. 

($\supseteq$) Let $w \in (\bigcup_{i \in J} X_i)^{-1}$. Hence, $\bigcup_{i \in J}X_i \in \mathcal{F}_{w}$. By the partition condition, there is $k \in J$ such that $X_k \in \mathcal{F}_{w}$. Hence, $w \in X^{-1}_{k} \subseteq \bigcup_{i \in J}X^{-1}_{i}$. 

Regarding the whole space: for each $w \in W$, $\mathcal{F}_{w} \neq \emptyset$. Hence, for each $w \in W$ there is $X_{w}$ such that $X_{w} \in \mathcal{F}_{w}$, i.e. $w \in X_{w}^{-1}$. Thus $W \subseteq \bigcup_{w \in W} X_{w}^{-1} = (\bigcup_{w \in W}X_{w})^{-1}$ and the other inclusion is trivial. Finally, $W$ is $\tau$-open. 

Now let us go to the pointwise equivalency. We shall prove only the modal case which should be written as such: $w \Vdash_\mu \bullet \varphi \Leftrightarrow w \Vdash_\tau \Box \varphi$. Assume that $w \Vdash_\mu \bullet \varphi$. Hence, there is $X \in \mathcal{F}_{w}$ such that for any $v \in X^{-1}$, $v \Vdash_\mu \varphi$. By induction hypothesis, $v \Vdash_\tau \varphi$. But $X^{-1} \in \tau$, so we can say that there is $Z \in \tau$ (namely, $Z = X^{-1}$) such that $w \in Z$ and for any $v \in Z$, $v \Vdash_\tau \varphi$. Thus $w \Vdash_\tau \Box \varphi$. 

Now suppose that $w \Vdash_\tau \Box \varphi$. Hence, there is $Z \in \tau$ such that $w \in Z$ and for any $v \in Z$, $v \Vdash_\tau \varphi$. By induction, $v \Vdash_\mu \varphi$. But there must be $X \in \mu$ such that $Z = X^{-1}$. Thus $w \in X^{-1}$, i. e. $X^{-1} \in \mathcal{F}_{w}$. This means that $w \Vdash_\mu \bullet \varphi$. 

\end{proof}
The next theorem is much simpler to prove: 
\begin{tw} Assume that $M' = \langle W', \tau, V_\tau \rangle$ is an \sgt-model. Then there exists pointwise equivalent i.f.s. \bgtf-model $M = \langle W, \mu, \mathcal{F}, V_\mu \rangle$. 
\end{tw}

\begin{proof}
(sketch)
It is easy to check that \sgt-model is just a special case of i.f.s. \bgtf-model. We should just identify $\mathcal{F}_{w}$ with the family of all open sets containing $w$ (for each $w \in W$). It is clear that superset condition holds. Obviously, $\mathcal{F}_{w} \neq \emptyset$. Concerning union partition condition, we may take an arbitrary $k \in J$ because $w$ is in each of its open neighbourhoods. Note that in \sgt-frame $X = X^{-1}$ for any $X \subseteq \bigcup \tau = W$.  
\end{proof}

\section{\gtff-models}
\label{fff}
In this section we shall slightly change the definition of generalized topological frame. Moreover, we shall work with two modal operators.

\begin{df}
We define \gtff-model as a quintuple $M_\mu = \langle W_\mu, \mu, \mathbf{f}, \mathcal{N}, V_\mu \rangle$ such that $\mu$ is a generalized topology on $W$, $V$ is a valuation from $PV$ to $P(W)$ and $W$ consists of two separate subsets $Y_1$ and $Y_2$:

\begin{enumerate}
\item If $w \in Y_1$, then we link $w$ with certain $v \in \bigcup \mu$ (by means of a function $\mathbf{f}$, i.e. $\mathbf{f}$ is a function from $Y_1$ into $\bigcup \mu$).

\item If $w \in Y_2$ then we associate $w$ with certain $\mathcal{N}_{w}$, hence $\mathcal{N}$ is a function from $Y_2$ into $P(P(W))$.

\end{enumerate}
\end{df}

As for the forcing of modal formulas, the following interpretation is given:

\begin{df}
In each \gtff-model $M_\mu = \langle W_\mu, \mu, \mathbf{f}, \mathcal{N}, V_\mu \rangle$, for any $w \in W$ and for any $\varphi$: 

\begin{enumerate}
\item $w \Vdash_\mu \Box \varphi \Leftrightarrow$ there exists $X \in \mu$ such that $w \in X$ and for any $v \in X$, $v \Vdash_\mu \varphi$.

\item $w \Vdash_\mu \blacksquare \varphi \Leftrightarrow$

\begin{enumerate}

\item There is $X \in \mu$ such that $\mathbf{f}(w) \in X$ and for any $v \in X$, $v \Vdash_\mu \varphi$; \quad \emph{iff} $w \in Y_1$.

\item $V(\varphi) \in \mathcal{N}_{w}$; \quad \emph{iff} $w \in Y_2$. 

\end{enumerate}
\end{enumerate}

\end{df}

Let us define system \giez as the following set of axiom schemes and rules: \cpc $\cup$ \{ $\maxx_\Box$, $\tax_\Box$, $\four_\Box$, $\rext_\Box$, $\rext_\blacksquare$, \mpon\}, where \cpc means all modal instances of the classical propositional tautologies. As we can see, $\blacksquare$-free fragment of this logic is just \mtf (with $\Box$ as modality). At the same time, $\Box$-free fragment is just the weakest modal logic (with respect to $\blacksquare$) in which we do not expect more than the rule of extensionality. 

Now we can define appropriate canonical model:

\begin{df}
\label{cangtff}
Canonical \gtff-model is a structure $\langle W_\mu, \mu, \mathbf{f}, \mathcal{N}, V_\mu \rangle$, where:

\begin{enumerate}

\item $W_\mu$ is the set of maximal theories of \giez.

\item $V_\mu$ is a function from $PV$ into $P(W)$ such that for any $q \in PV$, $V(q) = \{w \in W; q \in W\}$.

\item $\mu$ is a subfamily of $W$, namely the union of all \emph{basic sets}, where single basic set $\widehat{\Box \varphi} = \{w \in W; \Box \varphi \in w\}$ \footnote{In general $\widehat{\gamma} = \{w \in W; \gamma \in w\}$.}.

\item $Y_1$ is a set of all such theories from $W$ for which there is a theory $v \in \bigcup \mu$ such that for any formula $\varphi$ we have: $\Box \varphi \in v \Leftrightarrow \blacksquare \varphi \in w$.

\item $Y_2 = W \setminus Y_1$. 

\item $\mathbf{f}$ is a function from $Y_1$ into $\bigcup \mu$ such that $\mathbf{f}(w) = v$ where $v$ is as in (4).

\item $\mathcal{N}$ is a function from $Y_2$ into $P(P(W))$ defined as: $\mathcal{N}_{w} = \{ \widehat{\varphi}; \blacksquare \varphi \in w\}$. 

\end{enumerate}

\end{df}

Using Lindenbaum theorem and rule of $\blacksquare$-extensionality, we can easily prove the following two lemmas:

\begin{lem}
Let $W_\mu$ be a collection of all maximal theories of \giez and let $\{z \in W; \varphi \in z \} = \{z \in W ; \psi \in z\}$. Then $\varphi \rightarrow \psi \in \giez$. \end{lem}

\begin{lem}
\label{kanon}
In a canonical \gtff-model we have the following property: for each maximal theory $w$, if $\{z \in W; \varphi \in z\} \in \mathcal{N}_{w}$ and $\{z \in W; \varphi \in z\} = \{z \in W; \psi \in z\}$, then $\blacksquare \psi \in w$. 
\end{lem}

Now we can show the fundamental theorem referring to the question of completeness.

\begin{tw}
In a canonical \gtff-model, for any $\gamma$ and for any theory $w$, we have:

$w \Vdash_\mu \gamma \Leftrightarrow \gamma \in w$.

\end{tw}

\begin{proof}
The proof goes by induction. If $\gamma := \Box \varphi$ then we can repeat corresponding fragment from \cite{jarvi} (note that in this case we are interested only in worlds from $\bigcup \mu$). Hence, assume that $\gamma := \blacksquare \varphi$.

$(\Leftarrow)$

We have two options:

\begin{enumerate}

\item $w \in Y_1$. Assume that $\blacksquare \varphi \in w$. Then $\Box \varphi \in u = \mathbf{f}(w)$. Hence $u \in \widehat{\Box \varphi}$. This is a basic set of topology $\mu$, hence it belongs to $\mu$. By the axiom \tax $\Box$ (i.e. $\Box \varphi \rightarrow \varphi$) we state that $\widehat{\Box \varphi} \subseteq \widehat{\varphi}$. Hence there is $X = \widehat{\Box \varphi}$ such that $X \in \mu$, $u \in X$ and for all $v \in X$ we have $\varphi \in v$. By induction hypothesis it means that $v \Vdash_\mu \varphi$. Hence $u \Vdash_\mu \Box \varphi$. But this means that $w \Vdash_\mu \blacksquare \varphi$. 

\item $w \in Y_2$. Assume that $\blacksquare \varphi \in w$. Then (by the definition of $\mathcal{N}$) we have that $\{z \in W; \varphi \in z\} = \widehat{\varphi} \in \mathcal{N}_{w}$. By induction hypothesis $\widehat{\varphi} = \{z \in W; z \Vdash_\mu \varphi\} \in \mathcal{N}_{w}$. Hence $w \Vdash_\mu \blacksquare \varphi$. 

\end{enumerate}

$(\Rightarrow)$

Again we have two possibilities:

\begin{enumerate}

\item $w \in Y_1$. Assume that $w \Vdash_\mu \blacksquare \varphi$. Hence there is $X \in \mu$ such that $u = \mathbf{f}(w) \in X$ and for any $v \in X$, $v \Vdash_\mu \varphi$. If $X \in \mu$, then $X$ can be considered as a sum of several basic sets. Among them there must be set $\widehat{\Box \beta}$ (for certain formula $\beta$) such that $u \in \Box \beta$. Of course $\widehat{\Box \beta} \subseteq X$, hence for each $v \in \widehat{\Box \beta}$ we have $v \Vdash_\mu \varphi$. By induction hypothesis $\varphi \in v$. Hence $\widehat{\Box \beta} \subseteq \widehat{\varphi}$.

Now we can say that $\Box \beta \rightarrow \varphi$ is a theorem. Assume the contrary. Then there must be certain maximal, consistent theory to which both $\Box \beta$ and $\lnot \varphi$ belong. But we know that $\widehat{\Box \beta} \subseteq \widehat{\varphi}$. This is contradiction.

Now let us go back to the main part of the proof. We use \monot $\Box$ to prove implication $\Box \Box \beta \rightarrow \Box \varphi$. Using axiom \four $\Box$, we obtain $\Box \beta \rightarrow \Box \varphi$. This means that $u \in \widehat{\Box \beta} \subseteq \widehat{\Box \varphi}$. Hence, $\Box \varphi \in u$. But $u = \mathbf{f}(w)$. Then $\blacksquare \varphi \in w$. 

\item $w \in Y_2$. Let us assume that $w \Vdash_\mu \blacksquare \varphi$. Hence $\{z \in W; z \Vdash_\mu \varphi\} \in \mathcal{N}_{w}$. By induction hypothesis $\{z \in W; \varphi \in z\} \in \mathcal{N}_{w}$. By the definition of $\mathcal{N}$ and Lem. \ref{kanon} it means that $\blacksquare \varphi \in w$. 

\end{enumerate}

\end{proof}

Of course, \giez is not a very useful system because it does not give us any connection between two modalities. Hence, let us add $\Box \varphi \rightarrow \blacksquare \varphi$ to \giez (this new system will be called \giej). Let us now sketch briefly how we define appropriate subclass of frames (and corresponding canonical model). 

\begin{df}
We define \gtfi-model as a structure arising from \gtff-model but with the following stipulations:

\begin{enumerate}
\item $\bigcup \mu \subseteq Y_1$.

\item For any $w \in \bigcup \mu$, $\mathbf{f}$ is identity, i.e. $\mathbf{f}(w) = w$. 
\end{enumerate}
\end{df}

Clearly, our new axiom is satisfied in this environment. Let us check this fact: assume that $w \Vdash \Box \varphi$. By the definition, it means that there is $X \in \mu$ such that $w \in X$ and for any $v \in X, v \Vdash \varphi$. In particular, it means also that $w \in \bigcup \mu \subseteq Y_1$. Then $\mathbf{f}(w)=w$. This all allows us to say that $w \Vdash \blacksquare \varphi$.

In a canonical \gtfi-model we consider the set $W_{\blacksquare}$, defined as the set of all theories $w \in W$ such that for any $\varphi$: $\blacksquare \varphi \in w \Rightarrow \Box \varphi \in w$. Note that in each \giez-theory the converse implication is always true (because of the new axiom which has been added). Hence, for any $w \in W_{\blacksquare}$ we have: $\Box \varphi \in w \Leftrightarrow \blacksquare \in w$. Now let us define $\mu$ in a manner similar to the one from canonical \gtff-model, but assuming that $\mu$ is a subfamily of $W_{\blacksquare}$. Hence, $\bigcup \mu \subseteq W_{\blacksquare} \subseteq Y_1$. Finally, we assume that $\mathbf{f}$ is identity when restricted to $\bigcup \mu$. This leads us to the conclusion that \giej-logic is complete with respect to the class of all \gtfi-frames. 

Let us tell a few words about the properties of possible worlds in this framework. If $w \in W \setminus \bigcup \mu$ and there is $\varphi$ for which $w \Vdash \blacksquare \varphi \to \Box \varphi$, then $w \nVdash \blacksquare \varphi$. Of course in this situation we also have $w \Vdash \blacksquare \varphi \to \varphi$. If $w \in \bigcup \mu$, then for \emph{any} $\varphi$ we have $w \Vdash \blacksquare \varphi \to \Box \varphi$.

For convenience, let us introduce possibilities $\blacklozenge$ and $\Diamond$ defined as usual (in a dual manner to $\blacksquare$ and $\Box$, respectively). If $w \in Y_2$, then we cannot say that for \emph{any} $\varphi$, $w \Vdash \blacksquare \varphi \to \blacklozenge \varphi$ (we did not assume that if $X \in \mathcal{N}_{w}$, then $-X \notin \mathcal{N}_{w}$). If $w \in Y_1$, then for \emph{any} $\varphi$ the formula in question is satisfied. However, if $w \in Y_1 \setminus \bigcup \mu$, then it may be (for \emph{certain} $\psi$) that $w \nVdash \blacksquare \psi \to \Diamond \psi$ and $w \nVdash \blacklozenge \psi \to \Diamond \psi$. 

On the other hand, we are aware of the fact that such distinctions are not characterizations in a proper logical sense. In fact, it is easy to notice that \giej-logic is complete also with respect to this subclass of \gtfi-models which is defined by the condition $\bigcup \mu = Y_1$.


\begin{thebibliography}{90}

\bibitem{ahmet} H. Ahmet, T. Mehmet, \emph{Peritopological Spaces and Bisimulations}, Reports on Mathematical Logic, vol. 50 (2015), pp. 67-81. 

\bibitem{aiello} M. Aiello, J. van Benthem, G. Bezhanishvili, \emph{Reasoning about space: The modal way}, Journal of Logic and Computation, 13(6):889–920.

\bibitem{asso} T. Witczak, \emph{Generalized topological spaces with associating function}, \url{https://arxiv.org/pdf/1909.00460.pdf}

\bibitem{baskar}R. Baskaran, \emph{Generalized nets in generalized topological spaces}, Journal of Advanced Research in Pure Mathematics, 2011, pp. 1-7. 

\bibitem{space} J. v. Benthem, G. Bezhanishvili, \emph{Modal logics of space}, \url{http://www.illc.uva.nl/Research/Publications/Reports/PP-2006-08.text.pdf} 
 
\bibitem{csaszar} \'{A}. Cs\'{a}sz\'{a}r, \emph{Generalized topology, generalized continuity}, Acta Mathematica Hungarica, 96(4) (2002), 351 - 357.

\bibitem{doignon} J. P. Doignon, J. C. Falmagne \emph{Knowledge spaces and learning spaces}, \url{https://arxiv.org/abs/1511.06757}.

\bibitem{genopen} \'{A}. Cs\'{a}sz\'{a}r, \emph{Generalized open sets}, Acta Mathematica Hungarica, 75 (1997), 65-87.

\bibitem{separ} A. P. Dhana Balan, P. Padma, \emph{Separation spaces in generalized topology}, International Journal of Mathematics Research, Volume 9, Number 1 (2017), pp. 65 - 74.

\bibitem{halpern} J. Halpern, R. Pucella, \emph{Dealing with Logical Omniscence: Expresiveness and Pragmatics}, \url{https://arxiv.org/pdf/cs/0702011.pdf}

\bibitem{indrze} A. Indrzejczak, \emph{Labelled tableau calculi for weak modal logics}, Bulletin of the Section of Logic, Volume 36: 3/4 (2007), pp. 159 - 171.

\bibitem{jarvi} J. J\"{a}rvinen, M. Kondo, J. Kortelainen, \emph{Logics from Galois connections}, International Journal of Approximate Reasoning, vol. 49, Issue 3, November 2008, pages 595 - 606. 

\bibitem{korczak} E. Korczak-Kubiak, A. Loranty, R. J. Pawlak, \emph{Generalized (topological) metric spaces. From nowhere density to infinite games}, \url{http://dspace.uni.lodz.pl:8080/xmlui/bitstream/handle/11089/24747/89-104-korczak.pdf?sequence=1&isAllowed=y}.

\bibitem{pacuit} E. Pacuit, \emph{Neighborhood Semantics for Modal Logic}, Springer International Publishing AG 2017.


\bibitem{sarsak}M. S. Sarsak, \emph{New separation axioms in generalized topological spaces}, Acta Math. Hungar., 132 (3) (2011), 244 - 252.

\bibitem{soldano} H. Soldano, \emph{A modal view on abstract learning and reasoning}, 
Ninth Symposium on Abstraction, Reformulation, and Approximation, SARA 2011.

\bibitem{zand} M. R. Ahmadi Zand, R. Khayyeri, \emph{Two separation axioms in Generalized Topological Spaces}, The 45th Annual Iranian Mathematics Conference, August 26-29 2014. 

\end{thebibliography}
\end{document}